\documentclass[12pt,letterpaper]{article}

\usepackage{amssymb,amsfonts,amscd,amsthm}
\usepackage[all,arc]{xy}
\usepackage{enumerate, bbm}
\usepackage{mathrsfs}
\usepackage{tikz}
\usepackage[left=0.9in,top=0.9in,right=0.9in,bottom=0.9in]{geometry}
\usepackage{mathtools}
\usepackage{hyperref}
\usepackage{color}
\usepackage{graphicx}
\usepackage{enumitem} 
\graphicspath{ {TexImages/} }

\newtheorem{thm}{Theorem}[section]

\newtheorem{prop}[thm]{Proposition}
\newtheorem{lem}[thm]{Lemma}

\newtheorem{quest}[thm]{Question}

\newtheorem{defn}[thm]{Definition}

\hypersetup{
	colorlinks,
	citecolor=blue,
	filecolor=blue,
	linkcolor=blue,
	urlcolor=blue,
	linktocpage
}

\setenumerate[1]{label=\thesection.\arabic*.} 
\setenumerate[2]{label*=\arabic*.} 
\setlist[enumerate]{itemsep=2ex, topsep=2ex} 
\setlist[itemize]{itemsep=2ex, topsep=2ex}

\newcommand{\E}{\mathbb{E}}

\newcommand{\Gam}{\Gamma}

\newcommand{\ep}{\varepsilon}

\newcommand{\Om}{\Omega}

\newcommand{\del}{\delta}

\renewcommand{\l}{\left}
\renewcommand{\r}{\right}

\newcommand{\half}{\frac{1}{2}}

\newcommand{\sm}{\setminus}
\newcommand{\sub}{\subseteq}

\renewcommand{\c}[1]{\mathcal{#1}}
\renewcommand{\b}[1]{\mathbf{#1}}

\newcommand{\rec}[1]{\frac{1}{#1}}
\newcommand{\f}[2]{\frac{#1}{#2}}

\newcommand{\mr}[1]{\mathrm{#1}}

\newcommand{\supp}{\mr{supp}}

\newcommand{\greedy}{\c{S}_g}

\title{Online Card Games}
\author{Sam Spiro\footnote{Dept.\ of Mathematics, UCSD {\tt sspiro@ucsd.edu}. This material is based upon work supported by the National Science Foundation Graduate Research Fellowship under Grant No. DGE-1650112.}}
\date{\today}
\begin{document}
	\maketitle
\begin{abstract}
	Consider the following one player game.  A deck containing $m$ copies of $n$ different card types is shuffled uniformly at random.  Each round the player tries to guess the next card in the deck, and then the card is revealed and discarded.  It was shown by Diaconis, Graham, He, and Spiro that if $m$ is fixed, then the maximum expected number of correct guesses that the player can achieve is asymptotic to $H_m \log n$, where $H_m$ is the $m$th harmonic number.  
	
	In this paper we consider an adversarial version of this game where a second player shuffles the deck according to some (possibly non-uniform) distribution.  We prove that a certain greedy strategy for the shuffler is the unique optimal strategy in this game, and that the guesser can achieve at most $\log n$ expected correct guesses asymptotically for fixed $m$ against this greedy strategy.
\end{abstract}

\section{Introduction}
Consider the following one player game played by a player Guesser.  A deck with $n$ different card types with each card type appearing $m$ times is shuffled uniformly at random.  Each round, Guesser guesses what the card type of the next card in the deck is.  After the guess is made, the true identity of the card is revealed and then discarded from the deck.  This process continues until the deck is depleted, at which point the game ends and Guesser is given a point for each correct guess they made during the game.

This game is called the \textit{complete feedback model}, and it was motivated by several real world problems related to clinical trials \cite{BH,E} and to extrasensory perception experiments~\cite{DG}. We refer the interested reader to \cite{DGS} for more information on the history of this model, as well as to variants of the model which involve different levels of feedback.

A natural question to ask regarding the complete feedback model is how many points Guesser can expect to guess if they try to maximize or minimize their score.  To this end, if $\c{G}$ is a strategy in the complete feedback model\footnote{We formally define what a ``strategy'' is in Section~\ref{sec:optimal}.  Intuitively, it is a function which takes in the information which is currently known to Guesser and outputs a (possibly random) card type to guess.} and if $\pi$ is a way to shuffle a deck which consists of $m$ copies of $n$ different card types, we define $C_{m,n}(\c{G},\pi)$ to be the number of correct guesses made by Guesser in the complete feedback model if they use the strategy $\c{G}$ and if the deck is shuffled according to $\pi$.  Building on work of Diaconis and Graham~\cite{DG}, the following was proven by Diaconis, Graham, He, and Spiro~\cite{DGHS} and by He and Ottolini~\cite{HO}.  Here and throughout $\log n$ denotes the natural logarithm and $\Gam(x)$ denotes the gamma function.

\begin{thm}[\cite{DG,DGHS,HO}]\label{thm:old complete}
	For any fixed $m$ and $n$ tending towards infinity, we have
	\begin{align*}\max_{\c{G}}\E[C_{m,n}(\c{G},\pi)]&\sim  H_m \log n,\\ 
		\min_{\c{G}}\E[C_{m,n}(\c{G},\pi)]&\sim\Gam\l(1+\rec{m}\r) n^{-1/m},\end{align*}
	where the maximum and minimums run over all strategies $\c{G}$, the shuffling $\pi$ is chosen uniformly at random, and $H_m:=\sum_{i=1}^{m} i^{-1}$ is the $m$-th harmonic number.  Moreover, $\c{G}$ achieves this maximum/minimum if it guesses a most/least likely card each round.
\end{thm}
We note that the regime where $n$ is fixed and $m$ tends to infinity is studied in \cite{DG},  and in this case both the maximum and minimum scores are asymptotic to $m$.  

While the complete feedback model uses a deck that is shuffled uniformly, it is natural to consider other ways of shuffling the deck.  For example, Ciucu~\cite{C} investigated the model where dovetail shuffles are used, Pehlivan considered top to random shuffles~\cite{P}, and recently Liu~\cite{L} considered riffle shuffles.  In this paper we unify these approaches by allowing any distribution for shuffling the cards to be used.

More precisely, we consider the following two player game played by Guesser and Shuffler.  The game starts with an initial deck of cards.  Each round Shuffler chooses a card which is in the deck\footnote{In this formulation of the game, Shuffler chooses each card in an online fashion, possibly based on what Guesser has done in previous rounds.  Alternatively one could consider a model where Shuffler shuffles the deck at the start of the game and then draws each card according to this shuffling.  It is not difficult to show that the optimal expected scores under these two models are the same, but for concreteness we will only consider the online setting.}, and simultaneously Guesser guesses the card type of the selected card.  The card is then revealed and discarded.  The game continues in this way until the deck is completed.  We call this the \textit{Guesser-Shuffler game}.  We define $C_{m,n}(\c{G},\c{S})$ to be the number of correct guesses made by Guesser in this game if the deck starts with $m$ copies of $n$ different card types and if Guesser and Shuffler use the strategies $\c{G}$ and $\c{S}$, respectively.

For example, Shuffler could use the strategy ``shuffle the deck uniformly at random at the start of the game,  then sequentially pick the top card from the deck each round.''  Under this strategy the game reduces to the complete feedback model, and we know from Theorem~\ref{thm:old complete} that Guesser can play so that they get $H_m\log n$ points in expectation.  The central question we wish to answer is: can Shuffler significantly decrease this quantity if they use a different strategy?

To this end, we say that a strategy $\c{S}$ for Shuffler is a \textit{minimizing Shuffler strategy} if for all $m,n$ we have
\[\max_{\c{G}}\E[C_{m,n}(\c{G},\c{S})]=\min_{\c{S}'}\max_{\c{G}}\E[C_{m,n}(\c{G},\c{S}')],\]
where here the maximums and minimums range over all strategies for Guesser and Shuffler, respectively.  We similarly define what it means for a strategy to be a \textit{maximizing Shuffler strategy}.  To state the main result of this paper, we define the following.

\begin{defn}
The \textit{greedy strategy} $\greedy$ for Shuffler is defined as follows.  Each round, if the deck has $s$ different card types remaining in the deck, then Shuffler selects each card type with probability $s^{-1}$.
\end{defn}
For example, if at some point the deck consisted of 100 copies of one card type and 1 copy of another card type, then $\greedy$ would select each of these card types with probability $\half$.  One can easily show that the greedy strategy minimizes the probability that Guesser correctly guesses a card type in any given round, but it is far from obvious that it is a minimizing Shuffler strategy.  Indeed, in our example, if Shuffler chooses the card type with multiplicity 1, then Guesser can guarantee 100 correct guesses from later rounds.  Thus it might be the case that Shuffler should choose this low multiplicity card type with probability smaller than $\half$.  However, this turns out not to be the case.

\begin{thm}\label{thm:greedy}
	The greedy strategy $\greedy$ is both a minimizing and maximizing Shuffler strategy for the Guesser-Shuffler game.  Moreover, it is the unique minimizing Shuffler strategy.
\end{thm}
We note that there is no unique maximizing Shuffler strategy.  Indeed, once a card type $i$ has been exhausted from the deck, Guesser can minimize their score by guessing $i$ every round, and at this point the strategy that Shuffler uses is irrelevant.  However, it is possible to show that every maximizing Shuffler strategy must agree with the greedy strategy until a card type runs out of the deck, see Theorem~\ref{thm:greedyTech} for a precise statement of this form.

By using Theorem~\ref{thm:greedy} together with results from variants of the coupon collector and birthday problems, we can prove asymptotic bounds on how many correct guesses are made in the Guesser-Shuffler game when one player tries to minimize the number of correct guesses and the other tries to maximize it.

\begin{thm}\label{thm:main}
	For $m$ fixed and $n$ tending towards infinity, we have
	\begin{align}
	\min_{\c{S}}\max_{\c{G}}\E[C_{m,n}(\c{G},\c{S})]=\max_{\c{G}}\E[C_{m,n}(\c{G},\greedy)]&\sim \log n,\label{eq:mainMin}\\ 
	\max_{\c{S}}\min_{\c{G}}\E[C_{m,n}(\c{G},\c{S})]=\min_{\c{G}}\E[C_{m,n}(\c{G},\greedy )]&\sim \Gam\l(1+\rec{m}\r)(m!)^{1/m}\cdot  n^{-1/m}.\label{eq:mainMax}
	\end{align}
\end{thm}

In particular, comparing Theorem~\ref{thm:main} with Theorem~\ref{thm:old complete},  we see that Shuffler can significantly decrease or increase the expected number of correct guesses if they use the greedy strategy as opposed to shuffling the deck uniformly at random.

Lastly, we consider a variant of the Guesser-Shuffler game that we call the \textit{Restricted Matching Pennies game}.  This is a two player game which starts with both players being given an identical deck of cards.  Each round the two players simultaneously select a card from their deck, reveal their cards, and then discard them.  This continues until the two decks are depleted.   We define $M_{m,n}(\c{A},\c{B})$ to be the number of times that the selected cards matched during a round of the game if the deck starts with $m$ copies of $n$ different card types and if the two players use strategies $\c{A}$ and $\c{B}$.

One can view this game as a symmetric version of the Guesser-Shuffler game where Guesser is forced to guess each card type exactly $m$ times (with Guesser's deck recording the remaining guesses that they are allowed to make).  This can also be seen as a variant of the classical game ``Matching Pennies.''  In each round of this game, two players simultaneously reveal a penny from their hand, where one player gets a point if the coins match (i.e. if they are both heads or both tails), and the other gets a point if they do not match.  In contrast, the Restricted Matching Pennies game allows $n$ different types objects to be chosen from each round, and it adds the restriction that the number of times each object is chosen in total is fixed.

It turns out that in the Restricted Matching Pennies game that the optimal strategy for both players is to shuffle their decks uniformly at random, which is in sharp contrast to Theorem~\ref{thm:greedy}.

\begin{prop}\label{prop:twoDecks}
	For all $m,n$ we have
	\[\max_{\c{A}}\min_{\c{B}}\E[M_{m,n}(\c{A},\c{B})]=m.\]
	Moreover, if $\c{U}$ is the strategy of choosing each card from the remaining deck uniformly at random, then for any strategies $\c{A},\c{B}$ we have
	\[\E[M_{m,n}(\c{A},\c{U})]=\E[M_{m,n}(\c{U},\c{B})]=m.\]
\end{prop}

The rest of this paper is organized as follows.  In Section~\ref{sec:optimal} we prove Theorem~\ref{thm:greedy} showing that the greedy strategy is optimal in the Guesser-Shuffler game.  We then prove the asymptotic bounds of Theorem~\ref{thm:main} in Section~\ref{sec:asy} and the results for the Restricted Matching Pennies game in Section~\ref{sec:two}.  Concluding remarks and some open problems are given in Section~\ref{sec:concluding}.

\section{Proof of Theorem~\ref{thm:greedy}}\label{sec:optimal}

To prove our results, we first formally define some of the terms used in the introduction.

We say that $d$ is an \textit{$n$-vector} if it is a vector $(d_1,\ldots,d_n)$ of non-negative integers, and we think of $d_i$ as the number of card types of type $i$ left in a deck of cards.  We define the support of $d$ to be $\supp(d)=\{i:d_i>0\}$, and if $\b{X}$ is a random variable we define its support by $\supp(\b{X})=\{x:\Pr[\b{X}=x]>0\}$.   

We say that $\c{G}$ is a \textit{Guesser strategy} if it is a function which takes in nonzero $n$-vectors $d$ (for all $n$) and which outputs a random variable $\c{G}(d)$ whose support is in $[n]$.  We say that $\c{S}$ is a \textit{Shuffler strategy} if it is a function which takes in nonzero $n$-vectors $d$ (for all $n$) and outputs a random variable $\c{S}(d)$ whose support is in $\supp(d)$.  Intuitively, these definitions say that Guesser is allowed to guess any card type in $[n]$ while Shuffler must select a card type that is still in the deck.  For all nonzero $d$ we define the greedy strategy $\greedy$ by
\[\Pr[\greedy(d)=i]=|\supp(d)|^{-1}\hspace{1em} \forall i\in \supp(d).\]

When $n$ is understood, we let $\del_i$ be the $n$-vector which has a 1 in position $i$ and 0's everywhere else.  We recursively define the score $C(d,\c{G},\c{S})$ where $d$ is an $n$-vector and $\c{G}/\c{S}$ are Guesser/Shuffler strategies by having $C(d,\c{G},\c{S})=0$ if $d=(0,\ldots,0)$, and otherwise having \begin{equation}C(d,\c{G},\c{S})=1_{\c{S}(d)=\c{G}(d)}+C(d-\del_{\c{S}(d)},\c{G},\c{S}).\label{eq:Cdef}\end{equation}   
That is, if Shuffler chooses some $i=\c{S}(d)$, then a point is scored if Guesser chooses $i$ as well, and in either case the game continues with the deck $d-\del_i$ since one copy of $i$ has been removed.

For $d$ an $n$-vector and $\c{S}$ a Shuffler strategy, define \begin{align*}&f(d,\c{S})=\max_{\c{G}} \E[C(d,\c{G},\c{S})] &f(d)=\min_{\c{S}} f(d,\c{S})\\ &F(d,\c{S})=\min_{\c{G}} \E[C(d,\c{G},\c{S})] &F(d)=\max_{\c{S}} F(d,\c{S}),\end{align*} 
where the maximum and minimums run through all possible Shuffler and Guesser strategies as appropriate.  Intuitively, $f(d,\c{S})$ is the score of the game when Guesser tries to maximize their score given that the deck starts as $d$ and Shuffler uses strategy $\c{S}$, with the other functions having analogous interpretations. The fact that these maximums and minimums exist is due to the fact that every two player finite game has a Nash equilibrium.   If $\c{S}$ is such that $f(d)=f(d,\c{S})$ for all $n$-vectors $d$ of any length $n$, then we say that $\c{S}$ is a \textit{minimizing Shuffler strategy}, and similiary we define what it means for $\c{S}$ to be a \textit{maximizing Shuffler strategy}.  

The key lemma that we use throughout this section is the following recurrence relation.

\begin{lem}\label{lem:recurrence}
	Let $\c{S}$ be a Shuffler strategy.  If $d$ is an $n$-vector and $p_i=\Pr[\c{S}(d)=i]$, then
	\[f(d,\c{S})=\max_{i\in [n]} p_i+\sum_{j\in [n]} p_j f(d-\del_j,\c{S}),\]
	\[F(d,\c{S})=\min_{i\in [n]} p_i+\sum_{j\in [n]} p_j f(d-\del_j,\c{S}).\]
\end{lem}
\begin{proof}
	Given any Guesser strategy $\c{G}$, we have by \eqref{eq:Cdef} that
	\[\E[C(d,\c{G},\c{S})]=\Pr[\c{S}(d)=\c{G}(d)]+\sum p_j\cdot \E[C(d-\del_j,\c{G},\c{S})]\le \max p_i+\sum p_j\cdot f(d-\del_j,\c{S}),\]
	and taking a maximum over $\c{G}$ shows that 
	\[f(d,\c{S})=\max_{\c{G}} \E[C(d,\c{G},\c{S})]\le \max p_i+\sum p_j\cdot f(d-\del_j,\c{S}).\]
	Moreover, it is straightforward to prove by induction on $\sum d_i$ that the strategy $\c{G}$ which always guesses some $k$ with $p_k=\max p_i$ achieves this upper bound.  The proof for $F(d,\c{S})$ is essentially the same and we omit the details.
\end{proof}

Another key lemma we make use of is the following, which intuitively says that if $\c{S},\c{S}'$ are two strategies which act the same on $d$ as well as on any deck which is ``smaller'' than $d$, then these two strategies give the same expected score if we start with $d$.
\begin{lem}\label{lem:StratSteal}
	Let $\c{S},\c{S}'$ be two shuffler strategies and $d$ an $n$-vector such that $\c{S}(d')=\c{S}'(d')$ for all $n$-vectors $d'$ which satisfy $d_i'\le d_i$ for all $i$.  Then 
	\[f(d,\c{S})=f(d,\c{S}'),\hspace{2em} F(d,\c{S})=F(d,\c{S}').\]
\end{lem}
\begin{proof}
	We prove the more general result that, for $d,\c{S},\c{S}'$ as in the hypothesis, we have \begin{equation}\E[C(d,\c{G},\c{S})]=\E[C(d,\c{G},\c{S}')]\label{eq:StratSteal}\end{equation} 
	for all Guesser strategies $\c{G}$.  With this the stated result following by taking maximum/minimums with respect to this equality.  
	
	To prove this general result, assume for the sake of contradiction that there is some $d$ and strategies $\c{S},\c{S}',\c{G}$ that satisfy the hypothesis of the lemma and such that \eqref{eq:StratSteal} does not hold, and choose such an example with $\sum d_i$ as small as possible.  Observe that $d-\del_i$ for $i\in \supp(d)$ together with $\c{S},\c{S}'$ also satisfies the hypothesis of the lemma since $d$ does, so by the minimality of $d$ we must have $\E[C(d-\del_i,\c{G},\c{S})]=\E[C(d-\del_i,\c{S}',\c{G})]$ for all such $i$.  The desired equality now follows from the recursion \eqref{eq:Cdef}.
\end{proof}

The next result shows the intuitive fact that if, say, $d=(100,1)$ and Shuffler is trying to minimize their score, then it would be better for them to draw a 1 than a 2 (since the latter will guarantee 100 correct guesses from Guesser).

\begin{lem}\label{lem:smaller}
	Let $d$ be an $n$-vector.  If $d_i\ge d_j>0$, then \[f(d-\del_i)\le f(d-\del_j),\hspace{2em} F(d-\del_i)\ge F(d-\del_j).\]
\end{lem}
\begin{proof}
	Assume for contradiction that there exists such a $d$ with $f(d-\del_i)>f(d-\del_j)$, and choose this $d$ so that $\sum d_k$ is as small as possible.  Observe that if $d_i=d_j$, then the entries of $d-\del_i$ and $d-\del_j$ are permutations of one another, so by the symmetry of the problem we see that $f(d-\del_i)=f(d-\del_j)$.  Thus we can assume $d_i>d_j$, and in particular $d_i\ge 2$.

	Let $\c{S}$ be a minimizing Shuffler strategy and define $\c{S}'$ by having $\c{S}'(d-\del_i)=\c{S}(d-\del_j)$ and $\c{S}'(d')=\c{S}(d')$ for all other $d'$.  Observe that this is indeed a Shuffler strategy because 
	\[\supp(\c{S}'(d-\del_i))=\supp(\c{S}(d-\del_j))\sub \supp(d-\del_j)\sub \supp(d)=\supp(d-\del_i),\]
	where the last equality used $d_i\ge 2$.  Let $p_k=\Pr[\c{S}(d-\del_j)=k]$.  We claim that the following is a consequence of Lemmas~\ref{lem:recurrence} and \ref{lem:StratSteal}:
	\begin{equation}f(d-\del_i)-f(d-\del_j)\le f(d-\del_i,\c{S}')-f(d-\del_j,\c{S})=\sum_{k\in \supp(d-\del_j)} p_k(f(d-\del_i-\del_k)-f(d-\del_j-\del_k)).\label{eq:fromLemmas}\end{equation}
	In more detail, we have $f(d-\del_i)\le f(d-\del_i,\c{S}')$ because $f$ is a minimum over all possible Shuffler strategies, and $f(d-\del_j)=f(d-\del_j,\c{S})$ because $\c{S}$ is a minimizing Shuffler strategy.  This gives the inequality of \eqref{eq:fromLemmas}.  For the equality, we expand  $f(d-\del_i,\c{S}')$ and $f(d-\del_j,\c{S})$ using Lemma~\ref{lem:recurrence}, noting that $\max_{k\in [n]} \Pr[\c{S}'(d-\del_i)=k]=\max_{k\in [n]} p_k$ and
	\[f(d-\del_i-\del_k,\c{S}')=f(d-\del_i-\del_k,\c{S})=f(d-\del_i-\del_k),\]
	where the first equality uses Lemma~\ref{lem:StratSteal} and the second that $\c{S}$ is a minimizing strategy.  This proves \eqref{eq:fromLemmas}.
	
	Because $d_i>d_j$, we have $(d-\del_k)_i\ge (d-\del_k)_j$ for all $k$, so by the minimality of our choice of $d$ we have $f(d-\del_k-\del_i)\le f(d-\del_k-\del_j)$ for all $k$.  Applying this to each term of \eqref{eq:fromLemmas} gives $f(d-\del_i)\le f(d-\del_j)$, a contradiction to our assumption on $d$.  This implies the result for $f$, and an analogous argument gives the result for $F$.
\end{proof}

In order for the greedy strategy $\greedy$ to be a minimizing Shuffler strategy, it is necessary that $f(d-\del_i)\le 1+f(d-\del_j)$ for all $d$ and $i,j\in \supp(d)$.  Indeed if this failed for some $d$, then it would be better for Shuffler to deterministically choose $j$ (giving up a point) than to use any strategy which could draw $i$.  The following shows that a generalization of this condition is true if we inductively assume that the greedy strategy is optimal for all smaller decks.
\begin{lem}\label{lem:main}
	Let $d$ be an $n$-vector such that $f(d',\greedy)=f(d')$ for all $d'\ne d$ which satisfy $d'_k\le d_k$ for all $k\in [n]$.  Then for all $J\sub \supp(d)$ and $i\in \supp(d)$, we have
	\begin{align*}|J|f(d-\del_i)<1+\sum_{j\in J}f(d-\del_j),\\ 
	|J|F(d-\del_i)>1+\sum_{j\in J}F(d-\del_j).\end{align*}
\end{lem}
\begin{proof}
	Let $d,i,J$ be as in the hypothesis of the lemma and assume for contradiction that $|J|f(d-\del_i)\ge 1+\sum_{j\in J}f(d-\del_j)$.  Moreover, choose $d$ so that $\sum_k d_k$ is as small as possible, and then choose a set $J$ as small as possible such that this inequality holds for $d$.  By Lemma~\ref{lem:smaller}, if $d_i\ge d_{j'}$ for some $j'\in J$, then $f(d-\del_i)\le f(d-\del_{j'})$.  Thus \[|J|f(d-\del_i)\ge 1+\sum_{j\in J}f(d-\del_j)\implies |J\sm \{j'\}|f(d-\del_i)\ge 1+\sum_{j\in J\sm\{j'\}}f(d-\del_j),\] contradicting the minimality of $J$.  Thus we must have $d_j>d_i>0$ for all $j\in J$, and in particular $\supp(d-\del_j)=\supp(d)$ for all $j\in J$.  Let $s=|\supp(d)|$.
	
	First consider the case that $d_i\ge 2$, so $\supp(d-\del_k)=\supp(d)$ for all $k\in \supp(d)$.  Using this and the hypothesis $f(d-\del_k)=f(d-\del_k,\greedy)$ for all $k$, we find by Lemma~\ref{lem:recurrence} that
	\begin{align*}&|J|f(d-\del_i)-\sum_j f(d-\del_j)=|J|f(d-\del_i,\greedy)-\sum_j f(d-\del_j,\greedy)\\&=s^{-1}\sum_{k\in \supp(d)}\l(|J|f(d-\del_i-\del_k)-\sum_{j\in J} f(d-\del_j-\del_k)\r)<s^{-1}\sum_{k\in \supp(d)} 1=1,\end{align*}
	where the inequality used that each $d-\del_k$ satisfies the hypothesis of the lemma if $d$ does, and that $d$ is a counterexample to the lemma with $\sum d_p$ as small as possible.  This gives the desired inequality.
	
	Thus we can assume $d_i=1$. Consider the strategy $\c{S}'$ which agrees with $\greedy$ on $d'\ne d-\del_i$ and which has \[\Pr[\c{S}'(d-\del_i)=j]=\begin{cases}
		s^{-1}(|J|^{-1}+1) & j\in J,\\ 
		s^{-1} & j\in \supp(d)\sm ( J\cup \{i\}).
	\end{cases}\] 
Note that these probabilities sum to 1 and that this is indeed a Shuffler strategy.  By Lemmas~\ref{lem:recurrence} and \ref{lem:StratSteal} we have
	\begin{align*}|J|f(d-\del_i)&\le |J|f(d-\del_i,\c{S}')\\&=s^{-1}(1+|J|)+s^{-1}\sum_{j\in J}(1+|J|)f(d-\del_i-\del_j)+s^{-1}\sum_{k\in \supp(d)\sm(J\cup \{i\})} |J| f(d-\del_i-\del_k)\\ &=s^{-1}+s^{-1}|J|+s^{-1}\sum_{j\in J} f(d-\del_i-\del_j)+s^{-1}\sum_{k\in \supp(d)\sm\{i\}} |J|f(d-\del_i-\del_k).\end{align*}
	On the other hand, we have
	\begin{align*}\sum_{j\in J} f(d-\del_j)&=s^{-1}|J|+s^{-1}\sum_{j\in J} \l(f(d-\del_j-\del_i)+\sum_{k\in \supp(d)\sm \{i\}} f(d-\del_j-\del_k)\r)\\ &=s^{-1}|J|+s^{-1}\sum_{j\in J}f(d-\del_j-\del_i)+s^{-1}\sum_{k\in \supp(d)\sm \{i\}}\sum_{j\in J} f(d-\del_j-\del_k)\end{align*}
	and subtracting these two expressions gives
	\begin{align*}|J|f(d-\del_i)-\sum_{j\in J} f(d-\del_j)&\le s^{-1}+s^{-1}\sum_{k\in \supp(d)\sm \{i\}}\l(|J|f(d-\del_i-\del_k)- \sum_{j\in J}f(d-\del_j-\del_k)\r)\\&<s^{-1}+s^{-1}\cdot(s-1)=1,\end{align*}
	where again this inequality used the minimality of $d$.  This completes the proof for $f$, and a nearly identical argument works for $F$.
\end{proof}

With this we can prove our main result for this section, which is a slightly more precise version of Theorem~\ref{thm:greedy}.  To state our result, we say that an $n$-vector $d$ is \textit{fully-supported} if $\supp(d)=[n]$.
\begin{thm}\label{thm:greedyTech}
	Let $\c{S}$ be a Shuffler strategy and $\greedy$ the greedy strategy.  Then $\c{S}$ is a minimizing Shuffler strategy if and only if $S(d)$ has the same distribution as $\greedy(d)$ for all $n$-vectors $d$, and $\c{S}$ is a maximizing Shuffler strategy if and only if $\c{S}(d)$ has the same distribution as $\greedy(d)$ for all fully-supported $n$-vectors $d$.
\end{thm}

\begin{proof}
	Let $\c{S}$ be a minimizing Shuffler strategy, and assume for contradiction that there exists an $n$-vector $d$ such that $\c{S}(d)$ does not have the same distribution as $\greedy(d)$.  Choose such a $d$ so that $\sum d_i$ is as small as possible, which implies that Lemma~\ref{lem:main} applies to $d$.
	
	For all $k$, let $p_k=\Pr[\c{S}(d)=k]$.  By our choice of $d$, we have $p_k\ne p_{k'}$ for some $k,k'\in \supp(d)$.  By relabeling the card types, we can assume \[p_1=p_2=\cdots=p_i> p_{i+1}\ge p_{i+2}\ge\cdots \ge p_n\]
	for some $i\ge 1$ with $i+1\in \supp(d)$. Define $\ep=(p_1-p_{i+1})/(1+1/i)$ and let $\c{S}'$ be the strategy which has $\c{S}'(d')=\c{S}(d')$ for all $d'\ne d$, and for $d$ we have
	\[\Pr[\c{S}'(d)=k]=\begin{cases}
	p_1-i^{-1}\ep & k\le i,\\ 
	p_{i+1}+\ep & k=i+1,\\ 
	p_k & k\ge i+1.	
	\end{cases}\]
	Note that this is a Shuffler strategy because these probabilities add to 1 and $\Pr[\c{S}'(d)=k]>0$ only if $k\in \supp(d)$.  Also with this we have \[\max_{k}\Pr[\c{S}'(d)=k]=p_{i+1}+\ep=p_1-i^{-1}\ep.\] 
	
	Using Lemmas~\ref{lem:recurrence} and \ref{lem:StratSteal} and that $\c{S}$ is a minimizing Shuffler strategy, we find
	\[f(d,\c{S}')=p_1-i^{-1}\ep+\sum_{k\le i} (p_k-i^{-1}\ep)f(d-\del_k)+(p_{i+1}+\ep)f(d-\del_{i+1})+\sum_{k>i+1} p_k f(d-\del_k).\]
	Combining this with the recurrence for $\c{S}$ gives
	\[f(d,\c{S}')-f(d,\c{S})=-i^{-1}\ep-i^{-1}\ep\sum_{k\le i} f(d-\del_k)+\ep f(d-\del_{i+1})<0,\]
	where this last inequality used Lemma~\ref{lem:main} with $i+1$ and $J=\{1,2,\ldots,i\}$.  This contradicts $\c{S}$ being a minimizing Shuffler strategy, giving the desired result.
	
	The proof for the maximizing case is almost identical, and we briefly sketch the ideas for this argument.  If $d$ is not fully-supported then $F(d,\c{S})=0$ for all Shuffler strategies, so $\c{S}$ is a maximizing Shuffler strategy if and only if $F(d,\c{S})=F(d)$ for all fully-supported $d$. We assume for contradiction that $\c{S}$ is a maximizing Shuffler strategy and that there exists a fully-supported $d$ such that $\c{S}(d)$ does not have the same distribution as $\greedy(d)$.  This means there exists an $i$ with
	\[p_1=p_2=\cdots=p_i< p_{i+1}\le p_{i+2}\le\cdots \le p_n.\]
	We then define the strategy $\c{S}'$ exactly as before except that we use $\ep=-(p_1-p_{i+1})/(1+1/i)$.  Here $d$ being fully-supported is essential, as otherwise we might have $1\notin\supp(d)$ and that $\Pr[\c{S}'(d)=1]=-i^{-1}\ep>0$, so this would not be a Shuffler strategy.  From here the same analysis as before gives the result. 
\end{proof}
\section{Proof of Theorem~\ref{thm:main}}\label{sec:asy}
Given the work of the previous section, to prove Theorem~\ref{thm:main}, it suffices to determine how well the greedy strategy does against an optimal Guesser strategy. A key fact is that every ``reasonable'' Guesser strategy performs equally well against the greedy strategy in expectation.
\begin{lem}\label{lem:reduce}
	If $\c{G}$ is a Guesser strategy which always guesses a card type that is in the deck, then 
	\[\E[C_{m,n}(\c{G},\greedy)]=\max_{\c{G}'}\E[C_{m,n}(\c{G}',\greedy)].\]
	If $\c{G}$ is a Guesser strategy which, whenever possible, guesses a card type which is not in the deck, then
	\[\E[C_{m,n}(\c{G},\greedy)]=\min_{\c{G}'}\E[C_{m,n}(\c{G}',\greedy)].\]
\end{lem}
This result is straightforward to prove by using Lemma~\ref{lem:recurrence} and induction on $\sum d_i$, so we omit its proof.  With this lemma we can prove the bounds of Theorem~\ref{thm:main} by analyzing the performance of any specific Guesser strategy which is ``reasonable''.  In particular, we define $\c{G}^+$ to be some Guesser strategy which always guesses a card type which appears with the highest multiplicity in the deck.

Let $C_k$ denote the number of times Guesser correctly guesses a card type which had multiplicity $k$ in the deck.  Note that this is a random variable which depends on the choices of strategies $\c{G},\c{S}$ in the game, and also note that we always have $C_{m,n}(\c{G},\c{S})=\sum C_k$.  Let $t_k$ be the smallest integer such that at the start of round $t_k$, each card type has fewer than $k$ copies left in the deck.  Define $V_{k,\ell}$ to be the number of card types which have at least $\ell$ copies left in the deck at the start of round $t_k$.   For example, $V_{k,\ell}=0$ for all $k\le \ell$ and $V_{k,\ell}\le V_{k',\ell}$ for all $k\le k'$.  We also adopt the convention that $V_{m+1,m}=n$.

\begin{lem}\label{lem:T}
	If Guesser and Shuffler use the strategies $\c{G}^+,\greedy$ in the Guesser-Shuffler game with a deck which has $n$ card types of multiplicity $m$, then
	\[\E[C_m]=\log n+O(1),\]
	
	and for $k<m$ we have \[\E[C_k]\le \log(\E[ V_{k+1,k}])+O(1)\].
\end{lem}
Note that $V_{k+1,k}\ge 1$, so this logarithm is well defined.
\begin{proof}
	We claim that for all $k$,
	\[\E[C_k|V_{k+1,k}=r]=\sum_{x=1}^r x^{-1}.\]
	Indeed, after the last card type of multiplicity $k+1$ is drawn, Guesser (using strategy $\c{G}^+$) will continue to guess the remaining $r$ card types of multiplicity $k$ until they are depleted.  Because Shuffler uses the greedy strategy, the probability that the first of these card types drawn matches what Guesser guesses is exactly $r^{-1}$, the probability the second matches is $(r-1)^{-1}$ and so on; so the claim follows by linearity of expectation.  Thus
	\[\E[C_k]=\sum_{r\ge 1} \Pr[V_{k+1,k}=r]\cdot \sum_{x=1}^r x^{-1}=\sum_{r\ge 1} \Pr[V_{k+1,k}=r]\cdot (\log r+O(1))=\E[\log V_{k+1,k}]+O(1).\]
	
	The $k=m$ case of the lemma follows because $V_{m+1,m}=n$ deterministically, and the $k<m$ result follows from Jensen's inequality since $\log$ is a convex function.
\end{proof}  

We next show the following.

\begin{lem}\label{lem:coupon}
	If Shuffler uses the greedy strategy $\greedy$ in the Guesser-Shuffler game with a deck which has $n$ card types of multiplicity $m$, then for all $k<m$ we have
	\[\E[V_{k+1,k}]\le \E[V_{m,k}]\sim \f{(\log n)^k}{k!}\]
\end{lem}
\begin{proof}
	The first inequality follows from $V_{k+1,k}\le V_{m,k}$ since $k+1\le m$, so it suffices to prove the asymptotic result.  We do this by showing that $V_{m,k}$ has the same distribution as a random variable related to a variant of the coupon collector problem.
	
	The variant goes as follows: every day a household with $m$ brothers $b_1,\ldots,b_m$ receives a coupon $c$.  There are $n$ different types of coupons, each appearing with probability $1/n$. If $b_m$ does not already own a copy of $c$, then he adds it to his collection. Otherwise he gives it to his younger brother $b_{m-1}$ who keeps $c$ if he needs it, and otherwise he gives it to $b_{m-2}$ and so on.  Observe that $b_m$ will always be the first to obtain every type of coupon. Let $U_k$ denote the number of coupons that $b_k$ is missing when $b_m$ completes their collection.
	
	We claim that $U_k$ has the same distribution as $V_{m,k}$.  Indeed, we can form a coupling as follows.  Let $c_1,c_2,\ldots$ be the sequence of coupons that appear, and let $f(t)$ denote the $t$th coupon that was used in one of the $m$ collections (that is, this is the $t$th coupon after we throw away any coupon which has already appeared more than $m$ times).  We then define a deck of cards $\pi$ by letting $\pi_t=c_{f(t)}$ for all $t$.  Because each coupon type is equally likely to appear each day, it is not hard to show that $\pi$ is distributed the same as if Shuffler made it using the greedy strategy $\greedy$.   Moreover, having $b_k$ missing $r$ coupon types when $b_m$ finishes their set is equivalent to $\pi$ having $r$ card types with multiplicity $k$ after the last card with multiplicity $m$ is drawn in $\pi$, proving the claim.
	
	It was proven by Foata, Han, and Lass~\cite{FHL} that
	\[\E[U_k]\sim \f{(\log n)^k}{k!}\]
	for any fixed $k$, proving the result.  See also Adler, Oren and Ross~\cite{AOR} for a discussion of this result in English.  This completes the proof.
\end{proof}

With this we can now prove our main result.

\begin{proof}[Proof of Theorem~\ref{thm:main}]
	By Theorem~\ref{thm:greedy} and Lemma~\ref{lem:reduce}, we see that proving \eqref{eq:mainMin} is equivalent to showing that
	\[\E[C_{m,n}(\c{G}^+,\greedy)]\sim \log n,\]
	where $\c{G}^+$ is the strategy of guessing a card type which has highest multiplicity in the deck. This asymptotic result then follows by Lemmas~\ref{lem:T} and \ref{lem:coupon} together with the trivial bound $C_k\ge 0$ for all $k$.
	
	To prove \eqref{eq:mainMax}, let $\c{G}$ be any Guesser strategy which stops guessing card types that  are in the deck once a card type runs out in the deck.  by Theorem~\ref{thm:greedy} and Lemma~\ref{lem:reduce}, it suffices to prove
	\[\E[C_{m,n}(\c{G},\greedy)]\sim \Gam\l(1+\rec{m}\r)(m!)^{1/m}\cdot  n^{-1/m}.\]
	Let $\pi$ be the deck shuffled according to $\greedy$, and let $I_t$ be the indicator function which is 1 if $\c{G}$ correctly guesses $\pi_t$.  Let $T$ denote the largest index $t$ such that $\{\pi_1,\ldots,\pi_{t-1}\}$ does not contain a card type with multiplicity $m$.  Then $I_t=0$ if $t>T$, and more generally we have
	\[\Pr[I_t=1]=n^{-1}\cdot \Pr[T\ge t],\]
	as each card type is equally likely to appear in round $t\le T$.
	Because $C_{m,n}(\c{G},\greedy)=\sum I_t$, we see that
	\begin{equation}\E[C_{m,n}(\c{G},\greedy)]=n^{-1}\sum \Pr[T\ge t]=n^{-1}\cdot \E[T],\label{eq:C-}\end{equation}
	so it suffices to compute this expectation.
	
	We can interpret $T$ in terms of a variant of the birthday problem.  Namely, we can consider the following experiment: sample people with replacement until we find $m$ people which all have the same birthday.  If there are $n$ possible birthdays and each are equally likely, then the number of people we sample has the same distribution as $T$ (which can be proved using a coupling argument as in Lemma~\ref{lem:coupon}).  The expected value of this statistic for the birthday problem was computed by Klamkin and Newman~\cite{KN}, and their result implies
	\[\E[T]=\Gam\l(1+\rec{m}\r) (m!)^{1/m}\cdot n^{1-1/m},\]
	and plugging this into \eqref{eq:C-} gives the desired result.
\end{proof}

\section{Proof of Proposition~\ref{prop:twoDecks}}\label{sec:two}
We formally define the Restricted Matching Pennies game as follows.  We define $n$-vectors and $\del_i$ as in Section~\ref{sec:optimal}.   We say that $\c{A}$ is a strategy in this game if it is a function which takes in two $n$-vectors $a,b$ (in this order) and outputs a random variable $\c{A}(a,b)$ whose support lies in $\supp(a)$.  We recursively define the score $M(a,b,\c{A},\c{B})$ where $a,b$ are $n$-vectors with $\sum a_i=\sum b_i$ and $\c{A},\c{B}$ are strategies by being 0 if $a=b=(0,\ldots,0)$, and otherwise having \[M(a,b,\c{A},\c{B})=1_{\c{A}(a,b)=\c{B}(b,a)}+M(a-\del_{\c{A}(a,b)},b-\del_{\c{B}(b,a)},\c{A},\c{B}).\]
Our main goal for this section is to prove the following strengthening of Proposition~\ref{prop:twoDecks} which works for decks with arbitrary multiplicities.

\begin{prop}\label{prop:strongTwoDecks}
	Let $\c{U}$ be the strategy with $\Pr[\c{U}(a,b)=i]=a_i/\sum a_j$ for all $n$-vectors $a,b$.  Then for any strategies $\c{A},\c{B}$ and any two $n$-vectors $a,b$ with $\sum a_i=\sum b_i=N$, we have
	\[\E[M(a,b,\c{A},\c{U})]=\E[M(a,b,\c{U},\c{B})]=\sum_{i\in [n]}\f{a_ib_i}{N}.\]
\end{prop}

\begin{proof}
	For some slight ease of notation we define $f(a,b,\c{A})=\E[M(a,b,\c{A},\c{U})]$.  We prove by induction on $N$ that for all $n$-vectors $a,b$ with $\sum a_i=\sum b_i=N$, we have \[f(a,b,\c{A})=\sum_{i\in [n]}\f{a_ib_i}{N}\] for all strategies $\c{A}$, with the analogous result for $\c{B}$ being proven in exactly the same way.  
	
	The base case $N=0$ is trivial.  Assume we have proven the result up to some value $N$ and let $a,b$ be $n$-vectors with $\sum a_i=\sum b_i=N$.  Let \[f_{i,j}(a,b,\c{A})=\E[M(a,b,\c{A},\c{U})|\c{A}(a,b)=i,\ \c{U}(b,a)=j]\] and define $p_i=\Pr[\c{A}(a,b)=i]$.  Because $\c{U}$ guesses $j$ with probability $b_j/N$, we see that
	\begin{equation}f(a,b,\c{A})=\sum_{i} p_i \sum_j b_j N^{-1}\cdot  f_{i,j}(a,b,\c{A}).\label{eq:ij}\end{equation}
	
	By our inductive hypothesis and the definition of $M$, we have
	\[f_{i,i}(a,b,\c{A})=1+\f{(a_i-1)(b_i-1)}{N-1}+\sum_{k\ne i}\f{a_kb_k}{N-1}=(N-1)^{-1}(N-b_i-a_i+\sum_{k}a_kb_k),\]
	and for $i\ne j$ we have
	\[f_{i,j}(a,b,\c{A})=\f{(a_i-1)b_i}{N-1}+\f{a_j(b_j-1)}{N-1}+\sum_{k\ne i,j}\f{a_kb_k}{N-1}=(N-1)^{-1}(-b_i-a_j+\sum_k a_kb_k).\]
	Thus for any fixed $i$ we have
	\begin{align*}\sum_j b_j(N-1)\cdot f_{i,j}(a,b,\c{A})&=Nb_i+\sum_j b_j\l(-b_i-a_j+\sum_k a_kb_k\r)\\ &=Nb_i-Nb_i-\sum_j a_jb_j+N\sum_k a_kb_k=(N-1)\sum_j a_jb_j,\end{align*}
	where the second equality used that $\sum b_j=N$ in two places.	Plugging this expression into \eqref{eq:ij} and using $\sum_i p_i=1$ completes the inductive step in the proof of $f(a,b,\c{A})=\sum a_jb_j/N$ for all strategies $\c{A}$, proving the result.
\end{proof}

\section{Concluding Remarks}\label{sec:concluding}
\textbf{Strategies for Guesser.}  In Theorem~\ref{thm:greedy}, we showed that the greedy strategy is optimal for Shuffler in the Guesser-Shuffler game.  It is natural to ask what the optimal strategies are for Guesser.  Computations suggest that such strategies are complicated to describe, even for $n=2$ or $m=2$.  However, we do know some simple strategies for Guesser which perform close to best possible. 

Let $\c{G}^+$ be the strategy which uniformly at random chooses a card type which has the maximum number of copies left in the deck.  It is not hard to show that the expected score under this strategy is always at least the $n$th harmonic number $H_n\sim \log n$.  By \eqref{eq:mainMin}, we know that asymptotically Guesser can not do better than this in expectation.

Similarly one can come up with a strategy $\c{G}^-$ which achieves at most $\Theta_m(n^{-1/m})$ correct guesses in expectation.  This strategy goes through several phases, starting in Phase $m$.  If the strategy is currently in Phase $i$ with $i>0$ and there are at least $n^{(i-1)/m}$ card types in the deck with multiplicity smaller than $i$, then proceed to Phase $i-1$.  Otherwise uniformly at random guess a card type which appears with multiplicity $i$.  When the game reaches Phase 0 there is some card type which is no longer in the deck, and we guess this card type for the rest of the game.  It is not difficult to show that the expected score during each Phase is at most $n^{-1/m}$, so for $m$ fixed this gives about $n^{-1/m}$ correct guesses.

\textbf{Partial Feedback.}  The Guesser-Shuffler game is an adversarial version of the complete feedback model, and there are natural adversarial versions of other variants of this model.  For example, in the partial feedback model, the Guesser is only told whether their guess was correct or not each round (cf. this with the complete feedback model where the card is always revealed to the Guesser each round). Motivated by this, we define the \textit{online partial feedback game} to be the two player game where each round Shuffler chooses a card from the deck, Guesser guesses it, and then Guesser is told whether their guess is correct or not.  We similarly define the \textit{offline partial feedback game} where now Shuffler can freely shuffle the deck at the start of the game, but afterwards can not change the order of the cards.  

We note that in the Guesser-Shuffler game, the distinction between online and offline games were effectively irrelevant.  This is roughly because the state of the game (but not the score) at any point in time depends only on Shuffler's choices (i.e. the cards which they discarded from the deck), so Guesser's actions during the game should not influence how Shuffler shuffles.   However, under partial feedback, the state of the game also depends on Guesser's actions (i.e. this determines how much information Guesser has at any point), and seemingly Shuffler should be able to utilize this information during the game.

It was shown in \cite{DGHS} that for $n$ significantly large in terms of $m$, Guesser can not get more than $m+O(m^{3/4}\log m)$ points in expectation when the deck is shuffled uniformly at random and the Guesser is given partial feedback.  Thus in the adversarial models one can not hope for Guesser to get asymptotically more than $m$ correct guesses, which they can always guarantee by guessing the same card type each round.  When the deck is shuffled uniformly at random, it is known \cite{DGS} that Guesser can play so that they get $m+\Om(m^{1/2})$ correct guesses when $n$ is sufficiently large, but we do not know of any such strategy that works in the adversarial setting. 

\begin{quest}
	Is there a strategy for Guesser in either the online or offline partial feedback game which uses a deck with $m$ copies of $n$ different card types such that their expected score is ``significantly'' larger than $m$?
\end{quest}
The best strategy we are aware of in the offline version is for Guesser to uniformly at random choose some $i_1\in [n]$, then guess $i_1$ each round until they get $m$ correct guesses, then to uniformly at random choose some $i_2\in [n]\sm \{i_1\}$ and guess $i_2$ until $m$ correct guesses are made, and so on.  It is not difficult to see that Shuffler can arrange the deck ahead of time so that Guesser gets $m+1/2$ correct guesses in expectation under this strategy when $m>1$ (namely by shuffling the deck so that the last $n$ cards are all distinct).  However, in the online version of this game, Shuffler can guarantee that Guesser gets exactly $m$ correct guesses using this strategy when $m>1$.  

The best strategy we are aware of in the online setting is for Guesser to uniformly at random guess either 1 or 2 each round until they correctly guess either $m$ 1's or $m$ 2's, at which point they only guess the other card type for the rest of the game.  Note that if Guesser simply uniformly guessed either 1 or 2 each round, then it is easy to show that they get $m$ points in expectation.  By additionally ignoring cards that Guesser knows are not in the deck, one can show that they improve their score by roughly $2^{-m}$ in expectation.

\textbf{Other Games.}  One can form a common generalization of the Guesser-Shuffler and Restricted Matching Pennies games as follows.  For any integer $k$, define the $k$-Guesser-Shuffler game to be the same as the Guesser-Shuffler game with the restriction that Guesser can guess each card type at most $k$ times throughout the game.  For example, if each of the $n$ card types appears $m$ times, then the $mn$-Guesser-Shuffler game is equivalent to the Guesser-Shuffler game (since no restriction is placed on Guesser), and the $m$-Guesser-Shuffler game is equivalent to the Restricted Matching Pennies game.  

\begin{quest}
	What are the optimal strategies in the $k$-Guesser-Shuffler game when one player wishes to maximize and the other wishes to minimize the number of correct guesses made during this game?  How many correct guesses are made in expectation under these strategies?
\end{quest}
Our results in Theorems~\ref{thm:greedy}, \ref{thm:main}, and Proposition~\ref{prop:twoDecks} shows that the answers to this question vary significantly between $k=mn$ and $k=m$, and it would be interesting to know if there is some sort of interpolation between these results as $k$ varies.

The $k$-Guesser-Shuffler game can be viewed as playing multiple rounds of the classical Matching Pennies game where restrictions are placed on how many times each player may use each move.  It is natural to consider what happens for other games.   For example, the game Rock, Paper, Scissors has two players selecting one of Rock, Paper, or Scissors each round with Rock beating Scissors, Scissors beating Paper, and Paper beating Rock.  If a player Alice picks a move that beats the move of Bob, then Alice is given a point and Bob loses a point, and if they pick the same move both players are given 0 points.

\begin{quest}
	Define the game Semi-Restricted Rock, Paper, Scissors by playing $3m$ rounds of Rock, Paper, Scissors where one player must use each move exactly $m$ times (with no restrictions placed on the other player).  What are the optimal strategies in this game and what is the expected score under these strategies?
\end{quest}
We note that similar games have appeared in the manga Tobaku Mokushiroku Kaiji~\cite{Kaiji}, namely the games ``Restricted Rock, Paper, Scissors'' and ``E-Card.''

\textbf{Acknowledgments.}  The author would like to thank Jimmy He for pointing out minor typos in an earlier draft.

\bibliographystyle{abbrv}
\bibliography{OCG}

\end{document}